\documentclass[letterpaper, 10 pt, conference]{ieeeconf}  
\IEEEoverridecommandlockouts                              
\usepackage{graphics} 
\usepackage{epsfig} 
\usepackage{mathptmx} 
\usepackage{times} 
\usepackage{amsmath} 
\usepackage{amssymb}  
\usepackage{psfrag} 
\usepackage{subcaption,tikz} 
\usepackage{url}
\usepackage[english]{babel}
\usepackage{algorithm,algorithmic}
\newtheorem{theorem}{Theorem}[section]

\newtheorem{lemma}[theorem]{Lemma}

\newcommand{\CASE}[1]{\STATE \textbf{case} #1\textbf{:} \begin{ALC@g}}
\newcommand{\ENDCASE}{\end{ALC@g}}

\newcommand{\DEFAULT}{\STATE \textbf{default:} \begin{ALC@g}}
\newcommand{\ENDDEFAULT}{\end{ALC@g}}
\newcommand{\DEFAULTLINE}[1]{\STATE \textbf{default:} }

\title{\LARGE \bf
Permissive Barrier Certificates for Safe Stabilization Using Sum-of-squares *
}

\author{Li Wang, Dongkun Han, and Magnus Egerstedt$^\dagger$
\thanks{*The work by the first and third authors was sponsored by Grant No.
N0014-15-1-2115 from the U.S. Office for Naval Research, and the work of the second author was sponsored by the NASA Grant NNX16AH81A.}
\thanks{$^\dagger$Li Wang and Magnus Egerstedt are with the School of Electrical and Computer Engineering, Georgia Institute of Technology, Atlanta, GA 30332, USA, Email: {\tt\small \{liwang, magnus\}@gatech.edu}. Dongkun Han is with the Department of Aerospace Engineering,  University of Michigan, 1320 Beal Ave, Ann Arbor, MI 48109, USA. Email: {\tt\small dongkunh@umich.edu}. }
}

\begin{document}
\maketitle
\thispagestyle{empty}
\pagestyle{empty}

\begin{abstract}
Motivated by the need to simultaneously guarantee safety and stability of safety-critical dynamical systems, we construct permissive barrier certificates in this paper that explicitly maximize the region where the system can be stabilized without violating safety constraints. An optimization strategy is developed to search for the maximum volume barrier certified region of safe stabilization. The barrier certified region, which is allowed to take any arbitrary shape, is proved to be strictly larger than safe regions generated with Lyapunov sublevel set based methods. The proposed approach effectively unites a Lyapunov function with multiple barrier functions that might not be compatible with each other. Iterative search algorithms are developed using sum-of-squares to compute the most permissive, that is, the maximum volume,  barrier certificates. Simulation results of the iterative search algorithm demonstrate the effectiveness of the proposed method.
\end{abstract}

\section{INTRODUCTION} \label{sec:intro}
The controller design of safety critical dynamical systems, such as power systems, autonomous vehicles, industrial robots, and chemical reactors, requires simultaneous satisfaction of performance specifications and multiple safety constraints \cite{chesi2011domain, romdlony2016stabilization, chesi2008analysis}. Violation of safety constraints might result in system failures and injuries. The problem of safe stabilization, i.e., to stabilize the system while staying in a given safe set, poses a serious challenge to the controller design task.


The formal design for stabilization of nonlinear dynamical systems is oftentimes achieved using Control Lyapunov Functions (CLFs). Meanwhile, the safety of dynamical systems can be established with barrier certificates, which guarantee that the state of the system never enters specified unsafe regions \cite{prajna2007framework}. Barrier certificates are useful tools for safety verification in autonomous dynamical systems, see \cite{prajna2007framework, sloth2012compositional}, and references therein. While in control dynamical systems, barrier certificates can provably enforce dynamical safety constraints in various applications, e.g., adaptive cruise control \cite{xu2016correctness}, bipedal walking \cite{hsu2015control}, and multi-agent robotics \cite{wang2016multiobj, wang2017multidrone}. It is important to see that safe stabilization is not guaranteed in the intersection of the DoA and the safe region. Since the safety and stabilization objectives might be in conflict, a common control that satisfies both objectives does not necessarily exist \cite{romdlony2014uniting, xu2016control}.

In order to simultaneously achieve safety and stabilization of dynamical systems, a number of control design methods have been proposed in the literature to unite CLF with barrier certificates. For example, a barrier function was explicitly incorporated into the design phase of the CLF \cite{tee2009barrier, romdlony2014uniting}, which resulted in a single feedback control law if a ``control Lyapunov barrier function" inequality was satisfied. However, no feedback controller can be designed if these two objectives were in conflict. The condition for multiple barrier constraints to be compatible with each other was characterized in \cite{xu2016control, wang2016multiobj}. To deal with conflicting safety and stabilization objectives, an optimization based controller was developed in \cite{ames2014CBF} such that safety is strictly guaranteed while convergence to goal is relaxed when conflict occurs. 

In contrast to the aforementioned methods, this paper deals with the conflict between the safety and stabilization objectives by finding a region of safe stabilization, which is both contractive to the equilibrium and safe with respect to state constraints.  The region of safe stabilization is a subset of the intersection of the Domain of Attraction (DoA) and the safe region. Similar to the problem of estimating the DoA, it is usually not easy to obtain the exact region of safe stabilization for arbitrary dynamics. Thus, a good approximation algorithm to compute the region of safe stabilization is needed. For instance, safe stabilization funnels were designed to be sublevel sets of the Lyapunov function in \cite{majumdar2013control}. In this paper, we will present an approximation algorithm based on barrier certificates, which generates an estimate of the region that is strictly larger than the estimate based on Lyapunov sublevel set. In contrast to \cite{ames2014CBF,xu2016correctness}, no relaxation on the Lyapunov constraint is needed when it is united with the permissive barrier certificates, because the certificates and the Lyapunov constraint are always compatible by construction.

Estimating the region of safe stabilization is closely related to estimating the DoA of an equilibrium state, except for the extra consideration of safety constraints. Among the various DoA approximation methods proposed in the literature, methods using the subset of Lyapunov-like functions, such as quadratic Lyapunov functions \cite{tibken00cdc} and rational polynomial Lyapunov functions \cite{chesi13auto}, are proved to be effective \cite{parrilo00cit}. Further improvements on the Lyapunov sublevel set based methods are developed in \cite{henrion2014tac,valmorbida14acc, han2016estimating} to reduce the conservativeness with invariant sets. In this paper, the set invariance property is established with barrier certificates, which are allowed to take arbitrary shapes rather than the sublevel set of the Lypapunov function. This method leads to a non-conservative estimate of the DoA. 


 
The contribution of this paper is threefold. First, permissive barrier certificates that are guaranteed compatible with the Lyapunov function are synthesized to ensure simultaneous stabilization and safety enforcement of control dynamical systems. Second, iterative search algorithms to compute permissive barrier certified region of safe stabilization are developed based on sum-of-squares (SOS) programs. Third, barrier certificates are used to construct a non-conservative estimate of DoA by allowing the contractive region to take arbitrary shapes.

The rest of the paper is organized as follows. Preliminary results on barrier certificates are briefly revisited in Section \ref{sec:prelim}.  Barrier certificates for DoA estimation and safe stabilization are the topics of Sections \ref{sec:auto} and \ref{sec:control}, respectively. Conclusions are discussed in Section \ref{sec:conclude}.

\section{Preliminaries: Barrier Certificates for Dynamical Systems}\label{sec:prelim}
Preliminary results on barrier certificates are revisited here to set the stage for DoA estimation and safe stabilization. More specifically,  applications of barrier certificates in safety verification of autonomous systems and safe controller synthesis for control dynamical systems will be discussed. 

\subsection{Barrier Certificates for Autonomous Dynamical Systems}
Using the invariant set principle, barrier certificates can certify that state trajectories starting from an initial set $\mathcal{X}_0$ do not enter an unsafe set $\mathcal{X}_u$. Consider an autonomous system
\begin{equation}\label{eqn:sysauto}
\dot{x} = f(x),
\end{equation}
where $x\in\mathcal{X}$, and $f$ is locally Lipschitz continuous. Both $\mathcal{X}_0$ and $\mathcal{X}_u$ are subsets of $\mathcal{X}$. The barrier certificate \cite{prajna2007framework}, $h(x):\mathbb{R}^n\to\mathbb{R}$, needs to satisfy
\begin{eqnarray}
h(x)\geq  0, &\forall x\in \mathcal{X}_0, \nonumber\\
h(x)<  0, &\forall x\in \mathcal{X}_u, \nonumber \\
\frac{\partial h(x)}{\partial x}f(x) \geq 0, &  \forall x\in\mathcal{X}, \label{eqn:dh0}
\end{eqnarray}
so that the safety of the system is guaranteed. 


The condition (\ref{eqn:dh0}) is often too restrictive, since $h(x)$ has to be non-decreasing. A more permissive barrier certificate is presented in \cite{ames2014CBF,xu2016correctness}. The condition (\ref{eqn:dh0})  can be relaxed to
\begin{equation}
\frac{\partial h(x)}{\partial x}f(x) \geq -\kappa(h(x)),   \forall x\in\mathcal{X},\label{eqn:dh1}
\end{equation}
where $\kappa:\mathbb{R}\to\mathbb{R}$ is an extended class-$\kappa$ function (strictly increasing and $\kappa(0)=0$). Let the certified safe area be defined as $\mathcal{C}=\{x\in\mathcal{X}~|~h(x)\geq0\}$. By allowing the derivative of the barrier certificate to grow within the safe set $\mathcal{C}$, this barrier certificate can ensure the forward invariance of $\mathcal{C}$ in a non-conservative manner. 
\begin{figure}[h]
  \centering
  \resizebox{3.2in}{!}{\includegraphics{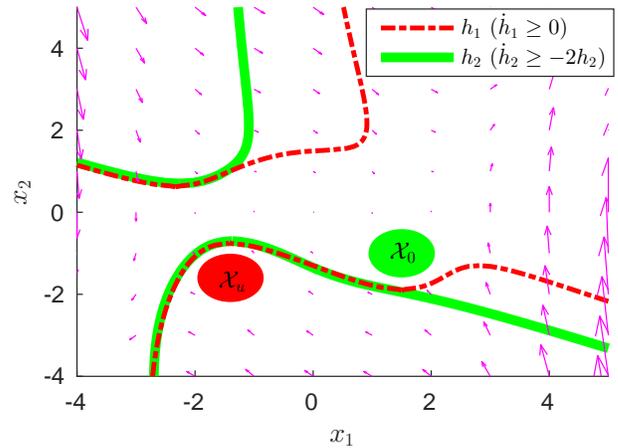}}
  \caption{Comparison of two types of barrier certificates. The barrier certified safe region based on (\ref{eqn:dh1}) (area between the solid green lines) is significantly larger than the safe region based on (\ref{eqn:dh0}) (area between the dashed red lines). }
  \label{fig:cpbarrier}
\end{figure}

The difference between these two types of barrier certificates can be illustrated with a simple example. Using the SOS technique described in \cite{prajna2007framework}, we can compute the certified safe regions for both barrier certificates. 

Consider a 2D autonomous dynamical system,
\begin{equation*}
\begin{bmatrix} \dot{x}_1 \\ \dot{x}_2 \end{bmatrix} 
= \begin{bmatrix} x_2 \\ -x_1+\frac{1}{3}x_1^3-x_2  \end{bmatrix}.
\end{equation*}
The initial and unsafe sets are specified as $\mathcal{X}_0=\{x~|~0.25 - (x_1-1.5)^2 - (x_2+1)^2\geq 0\}$ and $\mathcal{X}_u = \{x~|~0.25 - (x_1+1.4)^2 - (x_2+1.6)^2\geq 0\}$, respectively. Both types of barrier certificates can be illustrated in Fig. \ref{fig:cpbarrier}. The area of the barrier certified safe region generated with (\ref{eqn:dh1}) is much larger than (\ref{eqn:dh0}), which means that (\ref{eqn:dh1}) allows for a significantly more permissive safety certificate than (\ref{eqn:dh0}).

\subsection{Barrier Certificates for Control Dynamical Systems}
For a control-affine dynamical system
\begin{equation} \label{eqn:sys2}
\dot{x} = f(x) + g(x)u, 
\end{equation}
where $x\in \mathcal{X}$ and $u\in U$ are the state and control of the system, and $f$ and $g$ are both locally Lipschitz continuous. The safe set $\mathcal{C}=\{x\in\mathcal{X}~|~h(x)\geq0\}$ is defined as a superlevel set of a smooth function $h:\mathcal{X}\to \mathbb{R}$.

Barrier certificate can be designed to regulate the controller $u$, such that the safety constraint is never violated. The barrier certificate for control system is designed with control barrier functions (CBF). The function $h(x)$ is a CBF, if there exists an extended class-$\kappa$ function $\kappa$ such that
\begin{equation}
\sup_{u\in U} \left\{\frac{\partial h(x)}{\partial x}f(x) + \frac{\partial h(x)}{\partial x}g(x) u +\kappa(h(x))\right\}\geq 0, \forall x\in \mathcal{X}. \nonumber
\end{equation}
With $h(x)$, barrier certificates for (\ref{eqn:sys2}) are defined as
\begin{equation}
K(x) = \left\{u\in U~\middle|~\frac{\partial h(x)}{\partial x}f(x) + \frac{\partial h(x)}{\partial x}g(x) u +\kappa(h(x))\geq 0\right\}. \nonumber
\end{equation}
By constraining the controller $u$ in $K(x)$, the state trajectory will never leave the safe set $\mathcal{C}$ \cite{ames2014CBF,xu2016correctness}.

The stabilization task can be encoded into a control Lyapunov function (CLF) $V(x)$. Since a common control that satisfies both the CBF and the CLF does not necessarily exist, a typical way to unite the pre-designed CLF and CBF is to use a QP-based controller \cite{xu2016correctness,ames2014CBF,nguyen2016exponential}, i.e.,
\begin{equation}
\label{eqn:QPclfcbf}
 \begin{aligned}
& u^* &=  \:\: \underset{ u\in\mathbb{R}^{n}}{\text{argmin}}
   &\:\:J( u) + k_\delta\delta^2\\
 &  \text{s.t.}
 &  \frac{\partial V(x)}{\partial x}g(x) u &\leq -\frac{\partial V(x)}{\partial x}f(x) + \delta,    \\
 &
 &     -\frac{\partial h(x)}{\partial x}g(x) u &\leq \frac{\partial h(x)}{\partial x}f(x) + \kappa(h(x)),
 \end{aligned}
\end{equation}
where $\delta$ is a CLF relaxation factor, such that the non-negotiable safety constraint is always satisfied. However, simultaneous stabilization and safety enforcement are not guaranteed. In this paper, instead of relaxing the stabilization term, we will compute an estimate of the region of safe stabilization with permissive barrier certificates, such that both the stabilization and safety constraints are strictly respected.

\section{DoA Estimation with Barrier Certificates for Autonomous Dynamical Systems} \label{sec:auto}
Computing estimates of the region of safe stabilization is closely related to computing estimates of DoA, because both try to maximize the volume of interested region where certain matrix inequalities are satisfied. In this section, we will show that the DoA estimate derived with barrier certificates is strictly larger than the maximum contractive sublevel set of the Laypunov function. An iterative optimization algorithm based on SOS program is provided to numerically compute the most permissive barrier certificates for polynomial systems. Building upon the results developed in this section, permissive barrier certificates for safe stabilization will be presented in Section \ref{sec:control}. 

\subsection{Expanding Estimate of DoA with Barrier Certificates}
Assume the system (\ref{eqn:sysauto}) is locally asymptotically stable at the origin. Let $\psi(t;x_0)$ denote the state trajectory of the system (\ref{eqn:sysauto}) starting from $x_0$. The DoA of the origin is defined as the set of all initial states which eventually converge to the origin as time goes to infinity,
$$\mathcal{D}=\{x_0\in\mathcal{X}~|~\lim_{t\to\infty}\psi(t;x_0) = 0\}.$$ 

A commonly used method to estimate the DoA is to compute the sublevel set of a given Lyapunov function $V(x)$. This Lyapunov function should be positive definite, and its derivative should be locally negative definite. Let $\mathcal{V}(c) = \{x\in\mathcal{X}~|~V(x)\leq c\}$ be a sublevel set of $V(x)$. The largest inner estimate of the DoA using the sublevel set of the Lyapunov function can be computed with
\begin{equation}
\label{eqn:vcmaxopt}
 \begin{aligned}
& c^* = & \underset{ c\in\mathbb{R}}{\text{max}}
   &\:\:c\\
 &  \text{s.t.}
 &  -\frac{\partial V(x)}{\partial x}f(x)& > 0,\quad \forall x\in\mathcal{V}(c)\setminus \{0\}.
 \end{aligned}
\end{equation}
The estimate $\mathcal{V}(c^*)$ is straightforward to compute, but often conservative compared to invariant set based methods. This is because the shape of $\mathcal{V}(c^*)$ is restricted to the Lyapunov sublevel set.

Next, we will show that the estimate of DoA can be further expanded using barrier certificates and the given Lyapunov function. This is achieved by allowing the barrier certificates to take an arbitrary shape instead of the sublevel set of $V(x)$. The most permissive barrier certified region $\mathcal{C}=\{x\in\mathcal{X}~|~h(x)\geq 0\}$ can be computed as, 
\begin{equation}
\label{eqn:bmaxopt}
 \begin{aligned}
 \quad \quad h^*(x) =   \underset{ h(x)\in\mathcal{P}}{\text{argmax}} \quad
   &\mu(\mathcal{C})\\
  \text{s.t.}
  -\frac{\partial V(x)}{\partial x}f(x)&> 0,\quad &\forall x\in\mathcal{C}\setminus \{0\},\\
  \frac{\partial h(x)}{\partial x}f(x) &\geq - \kappa(h(x)),\quad &\forall x\in\mathcal{C},
 \end{aligned}
\end{equation}
where $\mu(\mathcal{C})$ is the volume of $\mathcal{C}$. The largest estimate of the DoA with barrier certificates is achieved with $\mathcal{C}^*=\{x\in\mathcal{X}~|~h^*(x)\geq 0\}$. By maximizing the volume of the barrier certified region, $\mathcal{C}^*$ is guaranteed to be larger than $\mathcal{V}(c^*)$. This fact can be shown with the following lemma.

\vspace{0.1in}
\begin{lemma} \label{lm:bdoa}
Given an autonomous system (\ref{eqn:sysauto}) that is locally asymptotically stable at the origin, the estimate of DoA with barrier certificates is no smaller than the estimate with the sublevel set of Lyapunov function, i.e., $\mu(\mathcal{V}(c^*))\leq \mu(\mathcal{C}^*)$.
\end{lemma}
\begin{proof}
The largest inner estimate of DoA using the sublevel set of a given Lyapunov function is $\mathcal{V}(c^*) = \{x\in\mathcal{X}~|~V(x)\leq c^*\}$. A candidate barrier certificate can be designed as $\bar{h}(x) = c^*-V(x)$, and the corresponding certified safe region is $\bar{\mathcal{C}}=\{x\in\mathcal{X}~|~\bar{h}(x)\geq 0\}$. The time derivative of $\bar{h}(x)$ is
\begin{equation*}
\frac{\partial \bar{h}(x)}{\partial x}f(x) = -\frac{\partial V(x)}{\partial x}f(x), \quad \forall x\in\bar{\mathcal{C}},
\end{equation*}
which is always nonnegative within $\bar{\mathcal{C}}$. By definition, $\bar{h}(x)$ is also nonnegative in $\bar{\mathcal{C}}$, i.e.,
\begin{equation*}
\frac{\partial \bar{h}(x)}{\partial x}f(x) \geq 0 \geq -\kappa(\bar{h}(x)), \quad \forall x\in\bar{\mathcal{C}},
\end{equation*}
which means $\bar{h}(x)$ is a valid barrier certificate and a feasible solution to (\ref{eqn:bmaxopt}). But $\bar{h}(x)$ is not necessarily the optimal solution. So we have $\mu(\mathcal{V}(c^*)) = \mu(\bar{\mathcal{C}}) \leq \mu(\mathcal{C}^*)$.
\end{proof}
\vspace{0.1in}

\textit{Remark $1$}:~　With \textit{Lemma} \ref{lm:bdoa}, (\ref{eqn:vcmaxopt}) can be reformulated into an optimization problem similar to (\ref{eqn:bmaxopt}), i.e.,
\begin{equation*}
\label{eqn:vcopt2}
 \begin{aligned}
& c^* = & \underset{ c\in\mathbb{R}}{\text{max}}
   &\:\:c\\
 &  \text{s.t.}
 &  -\frac{\partial V(x)}{\partial x}f(x)& > 0,\, &\forall x\in\mathcal{V}(c)\setminus \{0\}, \\
  &  
 &  \frac{\partial (c-V(x))}{\partial x}f(x)& \geq -\kappa(c-V(x)),\, &\forall x\in\mathcal{V}(c).
 \end{aligned}
\end{equation*}
We can see that (\ref{eqn:vcmaxopt}) also searches for a maximum barrier certificate. The shape of the certified region is constrained to be a sublevel set of $V(x)$. Since a specific shape of the certified region is not required, (\ref{eqn:bmaxopt}) is more permissive than (\ref{eqn:vcmaxopt}). In addition, $h(x)$ is allowed to decrease within the estimated DoA instead of monotone increasing.
\vspace{0.1in}

The fact that $\mathcal{C}^*$ is an inner estimate of the DoA can be established with the following theorem.
\vspace{0.1in}
\begin{theorem} \label{thm:bdoa}
Given an autonomous dynamical system (\ref{eqn:sysauto}) that is locally asymptotically stable at the origin, the estimate of the DoA with barrier certificates, $\mathcal{C}^*$,  is a subset of the true DoA $\mathcal{D}$. And $\mathcal{C}^*$ is guaranteed to be non-empty.
\end{theorem}
\begin{proof}
Given an arbitrary initial state $x_0\in\mathcal{C}^*$, the trajectory of the state $\psi(t; x_0), t\in[0,\infty),$ is guaranteed to be contained within $\mathcal{C}^*$, due to the forward invariance property of barrier certificates. 

By the construction of $\mathcal{C}^*$ in (\ref{eqn:bmaxopt}), $\frac{\mathrm{d} V(\psi(t; x_0))}{\mathrm{d} t}$ is negative definite for $\psi(t; x_0)\in\mathcal{C}^*$. Therefore, $V(\psi(t; x_0))$ is strictly decreasing along the trajectory $\psi(t; x_0), t\in[0,\infty),$ except at $0_n$. Since $V(x_0)$ is bounded and $0_n$ is the only equilibrium point in $\mathcal{C}^*$, we can get $\lim_{t\to\infty}\psi(t; x_0) = 0_n$. By the definition of the DoA, $x_0\in\mathcal{D}$ for any $x_0\in\mathcal{C}^*$,　which means $\mathcal{C}^*\subseteq \mathcal{D}$.

It is shown in \cite{chesi2011domain} that $\mathcal{V}(c^*)$ is non-empty. From Lemma \ref{lm:bdoa}, $\mu(\mathcal{V}(c^*)) \leq \mu(\mathcal{C}^*)$, thus $\mathcal{C}^*$ is also non-empty.
\end{proof}
\vspace{0.1in}

\subsection{Iterative Search of Permissive Barrier Certificates}
The optimization problem (\ref{eqn:bmaxopt}) is difficult to solve for general systems, since checking non-negativity is often computationally intractable \cite{papachristodoulou2002construction}. However, if non-negativity constraints are relaxed to SOS constraints, (\ref{eqn:bmaxopt}) can be converted to a numerically efficient convex optimization problem. To this end, we will restrict (\ref{eqn:sysauto}) to polynomial dynamical systems.

Let $\mathcal{P}$ be the set of polynomials for $x\in\mathbb{R}^n$. The polynomial $l(x)$ can be written in Square Matrix Representation (SMR) \cite{chesi2011domain} as $Z^T(x)QZ(x)$, where $Z(x)$ is a vector of monomials, and $Q\in\mathbb{R}^{k\times k}$ is a symmetrical coefficient matrix. A polynomial function $l(x)$ is nonnegative if $l(x)\geq 0, \forall x\in\mathbb{R}^n$. Furthermore, $p(x)$ is a SOS polynomial if $p(x)=\sum_{i=1}^{m}p_i^2(x)$ for some $p_i(x)\in\mathcal{P}$. $\mathcal{P}^\text{SOS}$ is the set of SOS polynomials. If written in SMR form, $p(x)$ has a positive semidefinite coefficient matrix $Q\succeq 0$. The trace and determinant of a square matrix $A\in\mathbb{R}^{n\times n}$ are $\text{trace}(A)$ and $\text{det}(A)$, respectively.

Since the proposed method is an under-approximation method, we would like to maximize the volume of $\mathcal{C}$ such that the best estimate of DoA can be achieved. However, this objective $\textrm{max}(\textrm{vol}(\mathcal{C}))$ is non-convex and usually cannot be described by an explicit mathematical expression. In order to solve this issue, a typical way adopted in the literature is to approximate the volume by using $\textrm{trace}(Q)$, where $h(x)=Z(x)^TQZ(x)$. In this paper, we would like to maximize $\textrm{trace}(Q)$ to get the largest $\mathcal{C}$ similar to \cite{chesi2011domain}.

To deal with nonnegativity constraints over semialgebraic sets, we will introduce the Positivestellensatz (P-satz).
\begin{lemma}(\cite{putinar1993positive})
For polynomials $a_1,\dots,a_m$, $b_1,\dots,b_l$ and $p$, define a set
	\begin{equation*}
	\begin{array}{rcl}
	\mathcal{B}&=&\{x\in \mathbb{R}^n: a_i(x)=0,~\forall i=1,\dots,m, \\
	&& b_i(x)\geq 0,~\forall j=1,\dots, l\}.
	\end{array}
	\end{equation*}
	Let $\mathcal{B}$ be compact. The condition $p(x)> 0,\forall x \in \mathcal{B}$ holds if the following condition is satisfied:
	\begin{equation*}
	\left\{
	\begin{array}{l}
	\exists r_1,\dots,r_m \in \mathcal{P},~ s_1,\dots, s_l \in \mathcal{P}^{\text{SOS}},\\
	p-\sum^{m}_{i=1}r_i a_i-\sum^{l}_{i=1}s_i b_i \in \mathcal{P}^{\text{SOS}}.
	\end{array}
	\right.
	\end{equation*}
\label{l:psatz}
\end{lemma}
This lemma provides an important perspective that any strictly positive polynomial $p(x)\in\mathcal{F}$ is actually in the cone generated by $a_i$ and $b_i$. Using the Real P-satz and the SMR form of $h(x)$, (\ref{eqn:bmaxopt}) can be formulated into a SOS program,
\begin{equation}
\label{eqn:boptsos0}
 \begin{aligned}
&& \underset{ \substack{h(x)\in\mathcal{P},~L_1(x)\in\mathcal{P}^{\text{SOS}}\\ L_2(x)\in\mathcal{P}^{\text{SOS}}}}{\text{max}} \quad
   \text{Trace}(Q)&\\
 &  \text{s.t.}
 &  -\frac{\partial V(x)}{\partial x}f(x) - L_1(x) h(x)&\in \mathcal{P}^{\text{SOS}},\\
 &
 &  \frac{\partial h(x)}{\partial x}f(x) + \gamma h(x) - L_2(x) h(x)&\in \mathcal{P}^{\text{SOS}}, 
 \end{aligned}
\end{equation}
where a linear function $\kappa(x)=\gamma x$ is adopted. The SOS program (\ref{eqn:boptsos0}) involves bilinear decision variables. It can be solved efficiently by splitting into several smaller SOS programs, which leads to the following iterative search algorithm. 

\vspace{0.1in}
\textit{Remark $2$}:~　Notice that (\ref{eqn:boptsos0}) requires an initial value of $h(x)$ to start with. From \textit{Lemma} \ref{lm:bdoa}, a good initial value can be picked as $\bar{h}(x)=c^*-V(x)$. This SOS program is guaranteed to generate a barrier certificate better than $\bar{h}(x)$.

\vspace{0.1in}
\textbf{Algorithm 1}: 

\textit{Step 1: Calculate an initial value for $h(x)$} 

Specify a Lyapunov function $V(x)$, and find $c^*$ using the bilinear search method, i.e.,
\begin{equation*}
\label{eqn:vcopt}
 \begin{aligned}
& c^* = & \underset{ c\in\mathbb{R}, L(x)\in \mathcal{P}^\text{SOS}}{\text{max}}
   &\:\:c\\
 &  \text{s.t.}
 &  -\frac{\partial V(x)}{\partial x}f(x) - L(x)(c-V(x)) & \in \mathcal{P}(x)^{\text{SOS}}.
 \end{aligned}
\end{equation*}
Set the initial value for $h(x)$ as $\bar{h}(x)=c^*-V(x)$.

\textit{Step 2: Fix $h(x)$, and search for $L_1(x)$ and $L_2(x)$} 

Using the $h(x)$ from previous step, we can search for $L_1(x)$ and $L_2(x)$ that give the largest margin on the barrier constraint. This is achieved by solving
\begin{equation*}
\label{eqn:boptsos}
 \begin{aligned}
&& \underset{ \substack{\epsilon\geq 0,~L_1(x)\in\mathcal{P}^{\text{SOS}}\\L_2(x)\in\mathcal{P}^{\text{SOS}}}}{\text{max}} \quad
   \epsilon &\\
 &  \text{s.t.}
 &  -\frac{\partial V(x)}{\partial x}f(x) - L_1(x) h(x)&\in \mathcal{P}^{\text{SOS}},\\
 &
 &  \frac{\partial h(x)}{\partial x}f(x) + \gamma h(x) - L_2(x) h(x) - \epsilon &\in \mathcal{P}^{\text{SOS}}.
 \end{aligned}
\end{equation*}

\textit{Step 3: Fix $L_1(x)$ and $L_2(x)$, and search for $h(x)$}

With $L_1(x)$ and $L_2(x)$ from previous step, a most permissive barrier certificate can be searched for. The barrier certificate is written in the SMR form $h(x)=Z(x)^TQZ(x)$. The most permissive barrier certificate is computed by maximizing the trace of $Q$,
\begin{equation*}
\label{eqn:boptsos}
 \begin{aligned}
&& \underset{ \substack{h(x)\in\mathcal{P}}}{\text{max}} \quad
   \text{trace}(Q) &\\
 &  \text{s.t.}
 &  -\frac{\partial V(x)}{\partial x}f(x) - L_1(x) h(x)&\in \mathcal{P}^{\text{SOS}},\\
 &
 &  \frac{\partial h(x)}{\partial x}f(x) + \gamma h(x) - L_2(x) h(x) &\in \mathcal{P}^{\text{SOS}}. 
 \end{aligned}
\end{equation*}
This searching process is terminated if $\text{trace}(Q)$ stops increasing, otherwise go back to \textit{Step 2}.
\vspace{0.1in}

\textit{Remark $3$}:~　In \textit{Step 2}, the common approach is to just search for feasible $L_1(x)$ and $L_2(x)$. However, there are multiple $L_1(x)$ and $L_2(x)$ available. By maximizing the margin $\epsilon$ of the barrier constraint, better options of $L_1(x)$ and $L_2(x)$ can be chosen.  This method will expand the feasible space of $h(x)$ for optimization in \textit{Step 3}, which can help speed up the optimization procedure.
\vspace{0.1in}

\subsection{Simulation Results for Autonomous Dynamical Systems}
The iterative search algorithm \textbf{1} is implemented on two examples of autonomous dynamical systems. In the simulation, the Matlab toolboxes SeDuMi, SMRSOFT \cite{chesi2011domain}, SOSTOOLS\cite{prajna2002introducing}, and YALMIP \cite{lofberg2005yalmip} are used for solving the semidefinite and SOS programming problems.

\vspace{0.1in}
\textit{Example $1$}:~ Given the two-dimensional autonomous system
\begin{equation*}
\begin{bmatrix} \dot{x}_1 \\ \dot{x}_2 \end{bmatrix} 
= \begin{bmatrix} x_2 \\ -x_1-x_2 -x_1^3  \end{bmatrix},
\end{equation*}
which has a locally stable equilibrium at the origin. A forth order Lyapunov function for this system can be picked as $V(x) = x_1^2+x_1x_2+x_2^2+x_1^4+x_2^4$. Using the sublevel set of $V(x)$, we can get the largest estimate of DoA as $$\mathcal{A}_1 = \{x\in\mathbb{R}^2~|~V(x)\leq 0.9759\}.$$ With the iterative search algorithm for barrier certificates, a larger estimate of DoA can be obtained as 
\begin{eqnarray*}
\mathcal{A}_2 = \{x\in\mathbb{R}^2~|~h(x) = 0.0428+0.0033x_1^2-0.1396x_1x_2\\+0.0206x_2^2  -0.0976x_1^4-0.0913x_2^4-0.0079x_1^3x_2\\+0.0061x_1x_2^3+0.0779x_1^2x_2^2 \geq 0\}. 
\end{eqnarray*}
For comparison under the same condition, the order of the barrier certificate is also restricted to be forth-order. As illustrated in Fig. \ref{fig:doaeg1}, the barrier certificate expands the estimate of DoA significantly.
\begin{figure}[h]
  \centering
  \resizebox{3.2in}{!}{\includegraphics{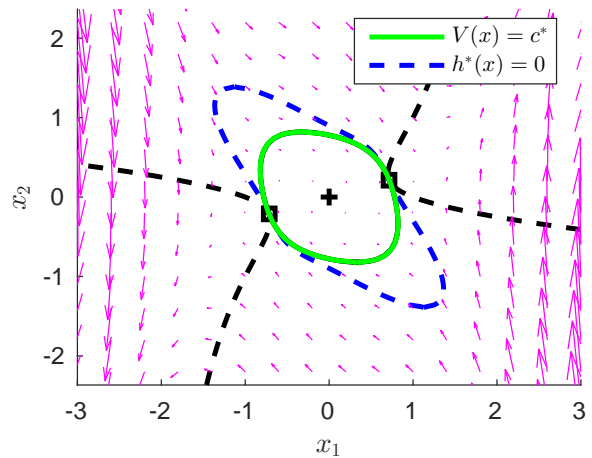}}
  \caption{Estimates of DoA for a two-dimensional autonomous dynamical system. The barrier certified DoA estimate (region enclosed by the dashed blue curve) is significantly larger than the Lyapunov sublevel set based DoA estimate (region enclosed by the solid green curve).}
  \label{fig:doaeg1}
\end{figure}
Note that the Lyapunov function in the example is randomly picked, one can also compute the maximal Lyapunov function \cite{chesi2011domain} and show that barrier certified DoA is larger as seen in \cite{wang2017safe}.

\vspace{0.1in}
\textit{Example $2$}:~ Consider the three-dimensional system
\begin{eqnarray*}
\begin{bmatrix} \dot{x}_1 \\ \dot{x}_2 \\ \dot{x}_3 \end{bmatrix} = 
\begin{bmatrix}
-x_1+x_2x_3^2 \\ -x_2 \\ -x_3
\end{bmatrix},
\end{eqnarray*}
which has a locally stable equilibrium at the origin. A Lyapunov function for this system can be picked as $V(x) = x_1^2+x_2^2+x_3^2$. The largest estimate of DoA based on the sublevel set of Lyapunov function is $$\mathcal{A}_1 = \{x\in\mathbb{R}^3~|~V(x)\leq 8\}.$$ With barrier certificates, the largest estimate of the DoA is 
\begin{eqnarray*}
\mathcal{A}_2 = \{x\in\mathbb{R}^3~|~h(x) = 7.9999-1.2828x_3^2-0.2850x_1^2 \\ -0.5652x_2^2-0.6685x_1x_2 \geq 0\}.
\end{eqnarray*}
The barrier certificate is restricted to the same order as $V(x)$. Both estimates of DoA are illustrated in Fig.\ref{fig:doa3Dauto}. Since both regions are ellipsoids, the volume of the estimated DoA can be analytically calculated. With the barrier certificate, the volume of the estimated region is increased by $\frac{\mu(\mathcal{A}_2)-\mu(\mathcal{A}_1)}{\mu(\mathcal{A}_1)}=297.4\%$.
\begin{figure}[h]
  \centering
  \resizebox{3.2in}{!}{\includegraphics{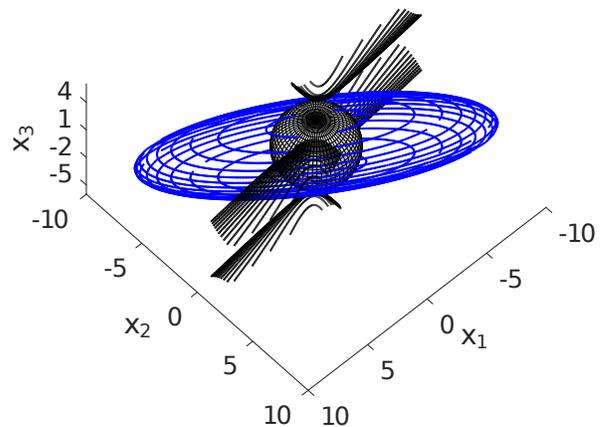}}
  \caption{Estimates of DoA for a three-dimensional autonomous dynamical system. The black and blue ellipsoids represent the largest estimate of DoA based on the Lyapunov function sublevel set and barrier certificates, respectively. }
  \label{fig:doa3Dauto}
\end{figure}

From these two examples, we can see that the barrier certificate based method provides a more permissive estimate of the DoA than the Lyapunov sublevel set based method.

\section{Safe stabilization of Control Dynamical Systems}\label{sec:control}
Permissive barrier certificates are developed in this section to maximize the estimated region of safe stabilization, where the system state is both stabilized and contained within the safe set. Based on the DoA estimation method for autonomous systems in section \ref{sec:auto}, the safe stabilization of control dynamical systems is addressed.

We will consider the safe stabilization problem described by (\ref{eqn:QPclfcbf}) for a locally stabilizable control-affine dynamical system (\ref{eqn:sys2}). Note that the locally stabilizable assumption ensures that an invariant and compact set for initial DoA estimation exists. Instead of relaxing the stabilization term with $\delta$ to resolve conflicts, we will synthesize a permissive barrier certificate with the maximum volume possible that strictly respects both the stabilization and safety constraints. This permissive barrier certificate can be found using 
\begin{equation}
\label{eqn:bumaxopt}
 \begin{aligned}
 \quad \quad h^*(x) =   \underset{ h(x)\in\mathcal{P}, u(x)\in\mathcal{P}}{\text{argmax}} \quad
   &\mu(\mathcal{C})\\
  \text{s.t.} \hspace{0.2in}
  -\frac{\partial V(x)}{\partial x}f(x) -\frac{\partial V(x)}{\partial x}g(x)u(x) &> 0, &\forall x\in\mathcal{C}\setminus \{0\},\\
  \frac{\partial h(x)}{\partial x}f(x) +\frac{\partial h(x)}{\partial x}g(x) u(x)  + \kappa(h(x))&\geq 0, &\forall x\in\mathcal{C},
 \end{aligned}
\end{equation}
where $\mu(\mathcal{C})$ is the volume of the certified safe region ($\mathcal{C}=\{x\in\mathcal{X}~|~h(x)\geq 0\}$). Note that (\ref{eqn:bumaxopt}) is a semi-infinite program that generates a feedback controller $u(x)$ for every $x\in\mathcal{C}$, while (\ref{eqn:QPclfcbf}) only products a point-wise optimal controller.    

To enforce the safety constraints, it is required that the barrier certified region is contained within the complement of the unsafe region, i.e., $\mathcal{C}\subseteq \mathcal{X}_u^c$. For generality, the unsafe region is encoded with multiple polynomial inequalities,
\begin{equation}\label{eqn:xu}
\mathcal{X}_u=\{x\in\mathcal{X}~|~q_i(x)< 0, ~\forall i\in\mathcal{M}\},
\end{equation}
where $q_i(x)$ are polynomials, and $\mathcal{M}=\{1,2,...,M\}$ is the index set of all the safety constraints. 

Similar to \textit{Lemma} \ref{lm:bdoa}, we can show that the region of safe stabilization estimated with barrier certificates is larger than the estimated region with Lyapunov sublevel set in \cite{majumdar2013control}.
\vspace{0.1in}
\begin{lemma} \label{lm:bdoau}
Given a dynamical control system (\ref{eqn:sys2}) that is locally stabilizable at the origin, the barrier certified region of safe stabilization estimate is no smaller than the estimated region of safe stabilization using sublevel set of the Lyapunov function, i.e, $\mu(\mathcal{V}(c^*))\leq \mu(\mathcal{C}^*)$.
\end{lemma}
\begin{proof}
Similar to Lemma \ref{lm:bdoa}.
\end{proof}
\vspace{0.1in}

In order to maximize the volume of the safe operating region, the barrier certificate is rewritten into SMR form, i.e., $h(x)=Z(x)^TQZ(x)$. Using the Real P-satz, the optimization problem (\ref{eqn:bumaxopt}) is formulated into a SOS program,
\begin{equation}
\label{eqn:boptsosu}
 \begin{aligned}
&  & \underset{ \substack{h(x)\in\mathcal{P}, ~u(x)\in\mathcal{P}\\L_1(x)\in\mathcal{P}^{\text{SOS}},~ L_2(x)\in\mathcal{P}^{\text{SOS}}\\J_i(x)\in\mathcal{P}^{\text{SOS}}, i \in\mathcal{M}}  }{\text{max}} \quad
   \text{Trace}(Q)&\\
 &  \text{s.t.}
 &  -\frac{\partial V(x)}{\partial x}(f(x) + g(x)u(x)) - L_1(x) h(x)&\in \mathcal{P}^{\text{SOS}},\\
 &
 &  \frac{\partial h(x)}{\partial x}(f(x) + g(x)u(x)) + \gamma h(x) - L_2(x) h(x)&\in \mathcal{P}^{\text{SOS}},  \\
 &
 & -h(x) + J_i(x)q_i(x)    \in \mathcal{P}^{\text{SOS}}, \forall i &\in \mathcal{M}. 
 \end{aligned}
\end{equation}
The optimal barrier certificate obtained by solving the SOS program (\ref{eqn:boptsosu}) is denoted by $h^*(x)$. The corresponding controller is $u^*(x)$. The following theorem shows that guaranteed safe stabilization can be achieved within the barrier certified region $\mathcal{C}^*$.

\vspace{0.1in}
\begin{theorem}
Given a dynamical control system (\ref{eqn:sys2}) that is locally stabilizable at the origin, a Lyapunov function $V(x)$, an unsafe region $\mathcal{X}_u$ in (\ref{eqn:xu}), and the solution $h^*(x)$ to (\ref{eqn:boptsosu}), for any initial state $x_0$ in $\mathcal{C}^*=\{x\in\mathcal{X}~|~h^*(x)\geq 0\}$, there always exists a controller that drives the system to the origin without violating safety constraints. 
\end{theorem}
\begin{proof}
Starting from any state $x_0\in\mathcal{C}^*$, the state trajectory of the system (\ref{eqn:sys2}) is denoted by $\psi(t; x_0)$ when the controller $u^*(x)$ from (\ref{eqn:boptsosu}) is applied.

By Real P-satz, the second constraint in (\ref{eqn:boptsosu}) implies that the barrier constraint in (\ref{eqn:bumaxopt}) is always satisfied, which ensures that the state trajectory $\psi(t; x_0)$ is always contained in $\mathcal{C}^*$.
Similarly, the first constraint in (\ref{eqn:boptsosu}) implies that $\frac{\mathrm d V(\psi(t; x_0))}{\mathrm d t}$ is always negative in $\mathcal{C}^*$ except at the origin. Thus $\lim_{t\to \infty}\psi(t; x_0)=0$.

The third constraint in (\ref{eqn:boptsosu}) ensures that ``if $-q_i(x)>0$, then $-h(x)>0$". Consider the contrapositive of this statement, we have ``if $h(x)\geq 0$, then $q_i(x)\geq 0$". This statement holds for any state $x\in\mathcal{C}^*$ and any safety constraint $i\in\mathcal{M}$, which means $\mathcal{C}^*\subseteq \mathcal{X}_u^c$. Because $\psi(t; x_0)$ is contained in $\mathcal{C}^*$, $\psi(t; x_0)$ is also contained in the safe space $\mathcal{X}_u^c$.

Combining these statements above, the controller $u^*(x)$ from (\ref{eqn:boptsosu}) will drive any state in $\mathcal{C}^*$ to the origin without violating any safety constraint.
\end{proof}
\vspace{0.1in}


\textit{Remark $4$}:~ With the generated permissive barrier certificates, it is guaranteed by construction that the QP-based controller (\ref{eqn:QPclfcbf}) is always feasible when $\delta$ is set to zero. This is because $u^*(x)$ is always a feasible solution for any $x\in\mathcal{C}^*$. The advantage of using a QP-based controller (\ref{eqn:QPclfcbf}) instead of $u^*(x)$ is that it minimizes the control effort by leveraging the part of nonlinear dynamics that contributes to stabilization.
\vspace{0.1in}

The optimization problem (\ref{eqn:boptsosu}) contains bilinear decision variables and requires a feasible initial barrier certificate. It can be split into several SOS programs and solved with the following iterative search algorithm.

\vspace{0.1in}
\textbf{Algorithm 2}: 

\textit{Step 1: Calculate an initial guess for $h(x)$}

Specify a Lyapunov function $V(x)$, and find $c^*$ using bilinear search
\begin{equation*}
\label{eqn:vcopt}
 \begin{aligned}
& c^* = & \underset{\substack{ c\in\mathbb{R}^+, ~u(x)\in \mathcal{P},~ L(x)\in\mathcal{P}^\text{SOS} \\ J_i(x)\in\mathcal{P}^\text{SOS}, ~i\in\mathcal{M}  } }{\text{max}}
   &\:\:c\\
 &  \text{s.t.}
 &  -\frac{\partial V(x)}{\partial x}(f(x)+g(x)u(x)) - L(x)(c-V(x)) & \in \mathcal{P}^{\text{SOS}}, \\
 &
 & -(c-V(x)) +J_i(x)q_i(x) \in \mathcal{P}^{\text{SOS}}, i &\in \mathcal{M}.
 \end{aligned}
\end{equation*}
With the result of the bilinear search, set the initial guess for the barrier certificate as $\bar{h}(x)=c^*-V(x)$, 

\textit{Step 2: Fix h(x), search for $u(x)$, $L_1(x)$, and $L_2(x)$}

Using the $h(x)$ from previous step, we can search for feasible $u(x)$, $L_1(x)$, and $L_2(x)$,  while maximizing the barrier constraint margin $\epsilon$.
\begin{equation*}
\label{eqn:boptsos}
 \begin{aligned}
&& \underset{ \substack{\epsilon\geq 0,~u(x)\in\mathcal{P}\\L_1(x)\in\mathcal{P}^{\text{SOS}},~L_2(x)\in\mathcal{P}^{\text{SOS}}}}{\text{max}} \quad
   \epsilon &\\
 &  \text{s.t.}
 & \hspace{-0.2in} -\frac{\partial V(x)}{\partial x}(f(x)+g(x)u(x)) - L_1(x) h(x)&\in \mathcal{P}^{\text{SOS}},\\
 &
 & \hspace{-0.2in} \frac{\partial h(x)}{\partial x}(f(x)+g(x)u(x)) + \gamma h(x) - L_2(x) h(x) - \epsilon &\in \mathcal{P}^{\text{SOS}}. 
 \end{aligned}
\end{equation*}

\textit{Step 3: Fix $u(x)$, $L_1(x)$, and $L_2(x)$, search for $h(x)$}

Rewrite the barrier certificate into SMR form $h(x)=Z(x)^TQZ(x)$. With the $u(x)$, $L_1(x)$, and $L_2(x)$ from the previous step, we can search for the maximum volume barrier certificate that respects all the safety constraints,
\begin{equation*}
\label{eqn:boptsos}
 \begin{aligned}
&& \underset{ \substack{h(x)\in\mathcal{P} \\ J_i(x)\in\mathcal{P}^\text{SOS}, ~i\in\mathcal{M} }}{\text{max}} \quad
   \text{trace}(Q) &\\
 &   \text{s.t.}
 &  -\frac{\partial V(x)}{\partial x}(f(x)+g(x)u(x)) - L_1(x) h(x)&\in \mathcal{P}^{\text{SOS}},\\
 &
 &  \frac{\partial h(x)}{\partial x}(f(x)+g(x)u(x)) + \gamma h(x) - L_2(x) h(x) &\in \mathcal{P}^{\text{SOS}}, \\
  &
 & -h(x) +J_i(x)q_i(x) \in \mathcal{P}^{\text{SOS}}, i &\in \mathcal{M}.
 \end{aligned}
\end{equation*}
Terminate if $\text{trace}(Q)$ stops increasing, otherwise go back to \textit{Step 2}.
\vspace{0.1in}

\textit{Remark $5$}:~　In \textit{Step 2}, the safety constraints $q_i(x)\geq 0, i\in\mathcal{M}$ do not need to be included. This is because $h(x)$ from previous step already satisfies these safety constraints. 
\vspace{0.1in}

\textit{Remark $6$}:~ To avoid unbounded control inputs, an additional constraint can be added to limit the magnitude of the coefficients of the polynomial controller $u(x)$.
\vspace{0.1in}

This iterative search algorithm is implemented on two control dynamical systems to achieve safe stabilization.

\textit{Example $3$}:~ Consider the simple two-dimensional mechanical dynamical system,
\begin{equation}\label{eqn:egu1}
\begin{bmatrix} \dot{x}_1 \\ \dot{x}_2 \end{bmatrix} = \begin{bmatrix}
x_2 \\ -x_1\end{bmatrix} + \begin{bmatrix} 0 \\ 1\end{bmatrix}u,
\end{equation}
where $x=[x_1, x_2]^T\in \mathbb{R}^2$ and $u\in\mathbb{R}$ are the state and control of the system. A Lyapunov function $V(x) = x_1^2+x_1x_2+x_2^2$ can be picked for the system.

The unsafe area is encoded with polynomial inequalities, $\mathcal{X}_u=\{x\in\mathbb{R}^2~|~q_i(x)< 0, i=1,2,3\}$, where
\begin{eqnarray*}
q_1(x) = (x_1-3)^2+(x_2-1)^2-1 < 0, \\
q_2(x) = (x_1+3)^2+(x_2+4)^2-1 < 0, \\
q_3(x) = (x_1+4)^2+(x_2-5)^2-1 < 0.
\end{eqnarray*}
The largest estimate of the region of safe stabilization with sublevel set of V(x) can be obtained as $$\mathcal{A}_1 = \{x\in\mathbb{R}^2~|~V(x)\leq 5.8628\}.$$ With the barrier certificate, this estimate can be enlarged to 
\begin{eqnarray*}
\mathcal{A}_2 =\{x\in\mathbb{R}^2~|~h(x)=0.5189-0.0669x_1-0.1196x_2\\ -0.0546x_1^2-0.0630x_1x_2-0.0294x_2^2\geq 0\}.
\end{eqnarray*}
For comparison purpose, the barrier certificate is restricted to be second order polynomial. These estimates are illustrated in Fig. \ref{fig:doau1}. By allowing the barrier certificate to be not centered around the equilibrium, the estimate of the region of safe stabilization is expanded significantly.
\begin{figure}[h]
  \centering
  \resizebox{3.2in}{!}{\includegraphics{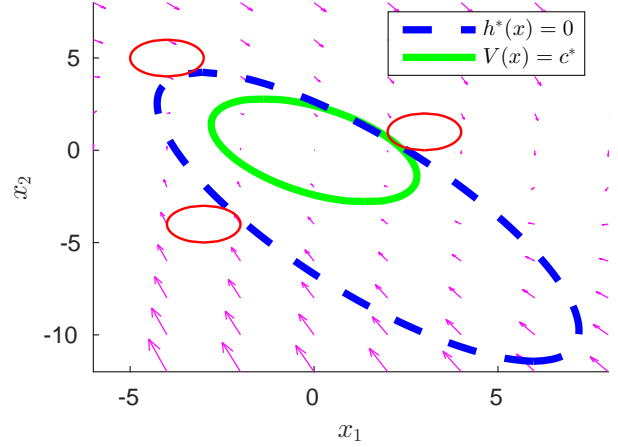}}
  \caption{Region of safe stabilization estimates for system (\ref{eqn:egu1}). The red circles represent unsafe regions. The magenta vector field represents the system dynamics when $u^*(x)$ is applied. The barrier certified region of safe stabilization (dashed blue ellipse) is significantly larger than the estimated region (solid green ellipse) with  Lyapunov sublevel set based methods. }
  \label{fig:doau1}
\end{figure}

\vspace{0.1in}
\textit{Example $4$}:~ Consider the three-dimensional system with multiple inputs,
\begin{equation}\label{eqn:egu2}
\begin{bmatrix} \dot{x}_1 \\ \dot{x}_2 \\ \dot{x}_3 \end{bmatrix} 
= \begin{bmatrix} x_2-x_3^2 \\ x_3-x_1^2+u_1 \\ -x_1 -2x_2 -x_3+x_2^3+u_2  \end{bmatrix},
\end{equation}
where $x=[x_1, x_2, x_3]^T\in \mathbb{R}^3$ and $u=[u_1, u_2]^T\in\mathbb{R}^2$ are the state and control of the system.
\begin{figure}[h]
  \centering
  \resizebox{3.2in}{!}{\includegraphics{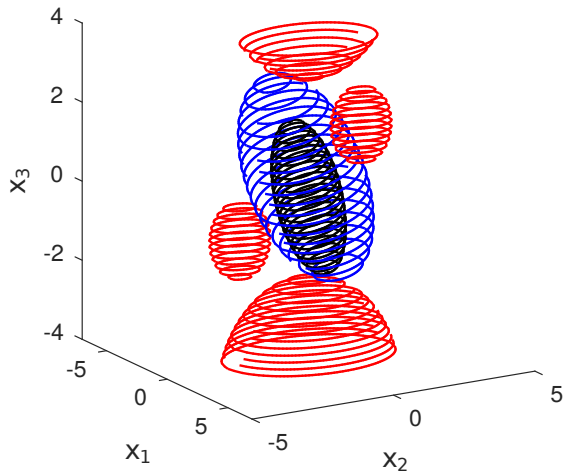}}
  \caption{Region of safe stabilization estimates for system (\ref{eqn:egu2}). The red spheres represent unsafe regions. The barrier certified region of safe stabilization (blue ellipsoid) is significantly larger than the region (black ellipsoid) obtained with Lyapunov sublevel sets.}
  \label{fig:doau2}
\end{figure}

A Lyapunov function for the system is picked to be $$V(x)=5x_1^2+10x_1x_2+2x_1x_3+10x_2^2+6x_2x_3+4x_3^2.$$ The unsafe region $\mathcal{X}_u=\{x\in\mathbb{R}^3~|~q_i(x) < 0, i=1,2,3,4\}$ is represented with polynomial inequalities
\begin{eqnarray*}
q_1(x) &=& (x_1-2)^2+(x_2-1)^2+(x_3-2)^2-1 < 0, \\
q_2(x) &=& (x_1+1)^2+(x_2+2)^2+(x_3+1)^2-1 < 0,　\\
q_3(x) &=& (x_1+0)^2+(x_2-0)^2+(x_3-6)^2-9 < 0,　\\
q_4(x) &=& (x_1+0)^2+(x_2+0)^2+(x_3+5)^2-9 < 0.
\end{eqnarray*}
The region of safe stabilization estimated with sublevel set of Lyapunov is $$\mathcal{A}_1 = \{x\in \mathbb{R}^3~|~V(x)\leq 13.0124\}.$$ Using the iterative search algorithm, the maximum permissive barrier certificate is
\begin{eqnarray*}
\mathcal{A}_2 = \{x\in\mathbb{R}^3~|~h(x)=114.3555+1.4686x_1+7.2121x_2\\ +19.8479x_3-24.5412x_3^2-14.7734x_1^2-26.0129x_1x_2\\ -15.5440x_1x_3-28.3492x_2^2-27.5651x_2x_3\geq 0\}.
\end{eqnarray*}
The results for region of safe stabilization estimates are shown in Fig. \ref{fig:doau2}. In both examples, the Lyapunov sublevel set search terminates as soon as the boundary of one safety constraint is reached, while the barrier certificate search terminates when all safety boundaries are touched. This also demonstrates the non-conservativeness of barrier certificates.


\section{Conclusions}\label{sec:conclude}
A theoretical framework to generate permissive barrier certified region of safe stabilization was developed in this paper to strictly ensure simultaneous stabilization and safety enforcement of dynamical systems. Iterative search algorithms using SOS programming techniques were designed to compute the most permissive barrier certificates. In addition, the proposed barrier certificates based method significantly expands the DoA estimate for both autonomous and control dynamical systems. The effectiveness of the iterative search algorithm was demonstrated with simulation results. 

Iterative algorithms were developed in this paper to cope with the non-convexity of the barrier certificated region maximization problems \eqref{eqn:boptsos0} and \eqref{eqn:boptsosu}. To get less conservative results, a promising way is to synthesize convex finite-dimensional LMIs rather than a bilinear matrix inequality using the moment theory and the occupation measure \cite{henrion2014tac}, to which our future efforts will be devoted.

\addtolength{\textheight}{-12cm}   


\bibliographystyle{abbrv}
\bibliography{mybib}
\end{document}